\documentclass{amsart}
\usepackage{amssymb,latexsym}

\numberwithin{equation}{section}
\newtheorem{thm}[equation]{Theorem}
\newtheorem{prop}[equation]{Proposition}

\newtheorem{cor}[equation]{Corollary}
\newtheorem{lem}[equation]{Lemma}
\theoremstyle{remark}

\newtheorem*{rem*}{Remark}
\theoremstyle{definition}
\newtheorem*{dfn}{Definition}
\newtheorem*{conj*}{Conjecture}
\newtheorem*{not*}{Notation}

\numberwithin{equation}{section}

\DeclareMathOperator{\Fitt}{Fitt}
\DeclareMathOperator{\GL}{GL}
\DeclareMathOperator{\Zen}{Z}

\newcommand{\Z}{\mathbb{Z}}
\newcommand{\R}{\mathbb{R}}

\newcommand{\normal}{\lhd}

\newcommand{\present}[2]{\langle #1 \mid \ #2 \rangle}
\newcommand{\gen}[1]{\langle #1 \rangle}

\begin{document}

\title{Unique product groups and congruence subgroups}

\author[W. Craig]{Will Craig}
\address{Department of Mathematics \\
University of Virginia \\
141 Cabell Drive, Kerchof Hall\\
P.O. Box 400137 \\
Charlottesville, VA 22904 \\
USA}
\email{wlc3vf@virginia.edu}

\author[P.A. Linnell]{Peter A. Linnell}
\address{Department of Mathematics \\
Virginia Tech \\
Blacksburg, VA 24061-1026 \\
USA}
\email{plinnell@math.vt.edu}


\begin{abstract}
We prove that a uniform pro-$p$ group with no nonabelian free
subgroups has a normal series with torsion-free abelian factors.
We discuss this in relation to unique product groups.  We also
consider generalizations of Hantzsche-Wendt groups.
\end{abstract}

\date{Wed Sep 23 19:02:52 EDT 2020}
\maketitle

\section{Introduction}

The group $G$ is called a unique product group if given two nonempty
finite subsets $X,Y$ of $G$, then there exists at least
one element $g\in G$
which has a unique representation $g = xy$ with $x\in X$ and $y\in
Y$.  It is easy to see that a unique product group is torsion free,
though the converse is not true in general
\cite{Rips87,Promislow88}.  There has been much interest recently in
determining which torsion-free groups are unique product groups
\cite{Carter14,Kionke16,Gruber15,Steenbock15}.
Original motivation for studying unique product
groups was the Kaplansky zero divisor conjecture, namely that if
$k$ is a field and $G$ is a torsion-free group, then $kG$ is a
domain.  This and other early results are described in \cite[\S
13.1]{Passman77}, which we summarize here.  The group $G$ is called a
two unique product group if given two nonempty finite subsets
$X,Y$ of $G$ with $|X| + |Y| \ge 3$,
there exist at least two elements in $G$ which have a unique
representation of the form $xy$ with $x\in X$ and $y\in Y$.
There is also the trivial units conjecture, namely that if $k$ is a
field and $G$ is a torsion-free group, then $kG$ has only trivial
units, i.e.\ the only units of $kG$ are of the form $ag$ with $0 \ne
a \in k$ and $g\in G$.  It is shown in \cite[Lemma 13.1.9]{Passman77}
that if $G$ is a two unique product group, then $kG$ has only trivial
units.  It is clear that if $G$ is a two unique product group, then
$G$ is a unique product group.  However Strojnowski \cite[Theorem
1]{Strojnowski80} showed that all unique product groups are two unique
product groups.  He also proved that if $G$ fails to be a unique
product group, then there exists a nonempty finite subset $X$ of $G$
such that no element of $G$ has a unique representation $xy$ with
$x,y\in X$, and it follows easily from Strojnowski's proof that the
finite subset $X$ can be chosen to be arbitrarily large.

If $H\lhd G$ are groups and $H$ and $G/H$ are unique product groups,
then $G$ is a unique product group \cite[13.1.8]{Passman77}.  Also it
is not difficult to see that direct limits of unique product groups
are unique product groups.  Since $\mathbb{Z}$ is a unique product
group, it follows that if there is a normal series $1=G_0 \lhd G_1
\lhd  \dots \lhd G_n = G$ such that $G_{i+1}/G_i$ is torsion-free
abelian for all $i$, then $G$ is a unique product group.

Let $p$ be a prime.  We shall use the notation $\mathbb{Z}_p$ for the
$p$-adic integers, $\mathbb{Q}_p$ for the $p$-adic numbers, and
$\GL_n(\mathbb{Z}_p)$ for the ring of invertible $n$ by $n$ matrices
over $\mathbb{Z}_p$.  Suppose $G$ is a uniform pro-$p$ group
\cite[Definition 4.1]{Dixon99}.  Good examples of pro-$p$ groups are
congruence subgroups.  Thus $\{A \in \GL_n(\mathbb{Z}_p) \mid A
\equiv 1 \mod p\}$ for $p$ odd, and $\{A \in \GL_n(\mathbb{Z}_2) \mid A
\equiv 1 \mod 4\}$ are uniform pro-$p$ groups \cite[Theorem
5.2]{Dixon99}.  Then $kG$ is a domain for all fields $k$ of
characteristic 0 or $p$ \cite[Theorem 1.3, remark to Proposition
6.4]{FarkasLinnell06}, \cite[Theorem 8.7.8]{Wilson98} and
\cite[Corollary 7.26]{Dixon99} (though it is
unknown whether this remains true for fields of nonzero
characteristic not equal to $p$).  We make the following conjecture:
\begin{conj*}
If $G$ is a uniform pro-$p$ group, then $G$ is a unique
product group.  
\end{conj*}
We cannot answer this conjecture, however we verify
that if $H$ is a virtually solvable subgroup of $G$, then $H$ is a
unique product group. In fact  we show that $H$ has an invariant series
$1 = H_0 \lhd  H_1 \lhd H_2 \lhd \dots \lhd H_n =H$ where $H_i \lhd H$
and  $H_{i+1}/H_i$ is torsion free abelian for
all $i$.  We will also show that if $H$ is virtually nilpotent, then
$H$ is nilpotent, and also if $H$ is virtually abelian, then $H$ is
abelian.  More precise statements are given in Theorems
\ref{T:VAbel}, \ref{T:VNil} and Corollary \ref{C:Nofree}.
The group $\langle x,y \mid x^{-1}y^2xy^2\rangle$, often called the
Hantzsche-Wendt group, was shown by Promislow \cite{Promislow88} to
be a nonunique product group.  This group is virtually abelian
but not abelian, so cannot be a subgroup of a uniform pro-$p$ group.
In the final section we consider some generalizations of the
Hantzsche-Wendt group.

Motivation for this paper is a result concerning amenable groups with
a locally invariant order \cite[Theorem 1.1]{Witte14}.
A locally invariant order $<$ on the group $G$ is a strict partial
order $<$ on $G$ with the property that for all $x,y \in G$ with
$y\ne 1$, either $xy > x$ or $xy^{-1}> x$.  Now a group with a
locally invariant order is a unique product group.  However if the
group happens also to be amenable, then $G$ has the much stronger
property of being locally indicable.

We would like to thank Andrzej Szczepa\'nski for helpful comments.

\begin{not*}
We shall use the notation $\Zen(G)$ for the center of the
group $G$, $\Fitt(G)$ for the Fitting subgroup of $G$
\cite[p.~46]{Wehrfritz73},
and $\mathbb{N}$ will denote the positive integers
$\{1,2,\dots\}$.
\end{not*}

\section{Uniform pro-$p$ groups}

\begin{prop} \label{P:uniqueroots}
Let $G$ be a uniform pro-$p$ group, let $0 \ne n\in \mathbb{Z}$, and
let $x,y\in G$.  If $x^n = y^n$, then $x=y$.
\end{prop}

\begin{proof}
This follows from \cite[Theorem 4.17]{Dixon99}.
\end{proof}

\begin{lem} \label{L:InductiveHelper}
Let $G$ be a virtually solvable residually finite $p$-group. Then $G$ is solvable.
\end{lem}

\begin{proof}
Let $H$ be a normal solvable subgroup of finite index in $G$, let
$\hat{G}$ denote the pro-$p$ completion of $G$, and let $K$ denote
the closure of $H$ in $\hat{G}$.  Then $K$ is a closed solvable
subgroup of finite index in $\hat{G}$, so $\hat{G}/K$ is a finite
pro-$p$ group and therefore $\hat{G}/K$ is a finite $p$-group.
It follows that $\hat{G}$ is solvable, which completes the proof
because $G$ is a subgroup of $\hat{G}$.
\end{proof}

Let $\mathcal{X}$ denote one of the following classes of groups:
\begin{itemize}
\item
All nilpotent groups of class at most $c$ for a fixed nonnegative
integer $c$.

\item
All nilpotent groups.

\item
All solvable groups.
\end{itemize}

Of course nilpotent groups of class at most 1 (case $c=1$ above) is
the class of abelian groups.

\begin{lem} \label{L:Uniform}
Let $G$ be a virtually $\mathcal{X}$-subgroup of a uniform group $U$.
Then there is a closed uniform virtually $\mathcal{X}$-subgroup $T$
of $U$ containing $G$.
\end{lem}
\begin{proof}
Since the closure of $G$ is a virtually $\mathcal{X}$-subgroup of $U$,
we may assume that $G$ is closed.  Then $G$ is a finitely generated
pro-$p$ group by \cite[Theorem 3.8]{Dixon99} and therefore it has a
characteristic uniform subgroup
$K$ of finite index by \cite[Corollary 4.3]{Dixon99}.  We now
consider the Lie algebra of $U$, as described in \cite[Chapter
4]{Dixon99}.  Then \cite[Proposition 4.31(i)]{Dixon99} shows that $K$
is a $\mathbb{Z}_p$-Lie subalgebra of $U$.  Let $T/K$ denote the
torsion subgroup of $U/K$; note that $T$ is a $\mathbb{Z}_p$-Lie
subalgebra of $U$.  Since $U$ is a finitely generated free
$\mathbb{Z}_p$-module by \cite[Theorem 4.17]{Dixon99}, we see that
$T/K$ is finite and in particular, that $T$ is a virtually
$\mathcal{X}$-subgroup of $U$ because $K$ is virtually $\mathcal{X}$.
Also \cite[Lemma 4.14(ii)]{Dixon99} shows that
$G\subseteq T$.  We now see from \cite[Proposition 7.15(i)]{Dixon99}
that $T$ is a closed uniform subgroup of $U$, and the result follows.
\end{proof}

\begin{thm} \label{T:VAbel}
Let $G$ be a virtually abelian subgroup of a uniform pro-$p$ group. Then $G$ is abelian.
\end{thm}

\begin{proof}
Suppose that $G$ is a subgroup of a uniform pro-$p$ group, and that $G$ is virtually abelian but not abelian.
By Lemma \ref{L:Uniform} we may assume that $G$ is uniform.
Let $B$ be a normal abelian subgroup of finite index in $G$.
Replacing $B$ with its closure in $G$, which is still a normal abelian
subgroup, we may assume that $B$ is closed, and then $G/B$ is a
finite pro-$p$ group, that is a finite $p$-group.  Now let $g \in G$
and $b\in B$.  Then $g^r \in B$ where $q$ is a power of $p$ and we
see that $g^qb=bg^q$.  Therefore $g^q = (bgb^{-1})^q$ and we deduce
from Proposition \ref{P:uniqueroots} that $g=bgb^{-1}$.  This shows that $B
\subseteq \Zen(G)$.  Now if $g,h \in G$, then $g^rh=hg^r$ where $r$ is
a power of $p$, hence $(hgh^{-1})^r = h^r$ and consequently $hgh^{-1}
= g$ by Proposition \ref{P:uniqueroots}.  Thus $G$ is abelian as required.
\end{proof}

\begin{thm} \label{T:VNil}
Let $c$ be a nonnegative integer and let $G$ be a virtually nilpotent
subgroup of class at most $c$ of a uniform pro-$p$ group. Then $G$ is
nilpotent of class at most $c$.
\end{thm}

\begin{proof}
By Lemma \ref{L:Uniform} we may assume that $G$ is uniform.
By hypothesis, there is a normal subgroup
$N$ of $G$ such that $N$ is nilpotent of class
at most $c$ and of finite index in $G$.
Replacing $N$ with its closure in $G$, we may assume that $N$ is
closed in $G$.
Since $\Zen(N)$ is closed in $G$ and since $G$ is powerful, we see
that $G/\Zen(N)$ is a powerful pro-$p$ group \cite[Chapter
3]{Dixon99}.
By \cite[Theorem 4.20]{Dixon99},
the elements of finite order in $G/\Zen(N)$ form a finite subgroup
$T/\Zen(N)$ which is normal in $G/\Zen(N)$ such that $G/T$ uniform.
Then $NT/T$ is a normal subgroup of finite index in $G/T$ and has
nilpotency class at most $c-1$.  Now since $T/\Zen(N)$ is finite, $T$
is virtually abelian and we see that $T$ is
abelian by Theorem \ref{T:VAbel}. Now let $t \in T$.  Then $t^q \in
\Zen(N)$ for some power $q$ of $p$.  Therefore if $h \in NT$, then
$(hth^{-1})^q = ht^qh^{-1} = t^q$ and since $T$ is a torsion-free
abelian group, we deduce that $hth^{-1} = t$.  This shows $T
\subseteq \Zen(NT)$.  Now let $t \in T$ and $g \in G$.  Then $g^r \in
N$ where $r$ is a power of $p$ and we see that $g^rt =
tg^r$, consequently $(tgt^{-1})^r = g^r$.  Therefore $tgt^{-1} = g$
by Proposition \ref{P:uniqueroots} and we conclude that
$t \in \Zen(G)$.  By induction $G/T$ is nilpotent of class at most
$c-1$, and the result follows.
\end{proof}

\begin{thm} \label{T:VSolvable}
Let $G$ be a virtually solvable subgroup of a uniform pro-$p$ group. Then
$\Fitt(G)$ is nilpotent and then $G/\Fitt(G)$ is torsion-free
abelian.
\end{thm}

\begin{proof}
By Lemma \ref{L:Uniform} we may assume that $G$ is uniform.
Let $F =\Fitt(G)$.
Using \cite[Theorem 7.19]{Dixon99}, we see that $G$ is a linear group
over $\mathbb{Q}_p$ and it now follows from
\cite[Theorem 8.2(ii)]{Wehrfritz73} that $F$ is nilpotent.
Since the closure of a nilpotent subgroup is nilpotent, we see that
$F$ is closed in $G$.  By \cite[Theorem 3.6]{Wehrfritz73}, $G/F$ is
virtually abelian.  Let $T/F$ denote the elements of finite order in
$G/F$.  Since $G/F$ is a powerful pro-$p$ group \cite[Chapter
3]{Dixon99}, $T/F$ is a finite normal subgroup of $G/F$ and $G/T$
is uniform \cite[Theorem 4.20]{Dixon99}.  By Theorem
\ref{T:VAbel}, $G/T$ is abelian and by Theorem \ref{T:VNil}, $T$ is
nilpotent.  The result follows.
\end{proof}

Applying the Tit's alternative, we obtain
\begin{cor} \label{C:Nofree}
Let $G$ be a subgroup of a uniform pro-$p$ group and suppose $G$
contains no nonabelian free subgroups.  Then
$\Fitt(G)$ is nilpotent and $G/\Fitt(G)$ is torsion-free
abelian.
\end{cor}

\begin{proof}
By \cite[Theorem 7.10]{Dixon99}, $G$ is a linear group over
$\mathbb{Q}_p$.  We now apply the Tit's alternative \cite[Corollary
10.17]{Wehrfritz73} to deduce that $G$ is virtually solvable.  The
result now follows from Theorem \ref{T:VSolvable}.
\end{proof}

\section{Hantzsche-Wendt groups}

\begin{dfn}
Define the combinatorial generalized Hantzsche-Wendt group $G_n$ by
\[G_n = \present{x_1,\dots,x_n}{x_i^{-1}x_j^2x_ix_j^2 \ \forall \ i \not = j}\]
\end{dfn}
Call any group isomorphic to some $G_n$ a CHW group for shorthand.
Note that the Hantzsche-Wendt group is isomorphic to $G_2$.  Of
course $G_0 =1$ and $G_1 \cong \mathbb{Z}$.

\begin{lem} \label{L:Free Abelian}
Let $A_n = \gen{x_1^2, x_2^2, \dots x_n^2}$, a subgroup of $G_n$.
Then $A_n \normal G_n$ and $A_n$ is free abelian on $\{x_1^2, \dots,
x_n^2\}$.
\end{lem}

\begin{proof}
The relations $x_i^{-1} x_j^2 x_i x_j^2 = 1$ imply that
$x_i^{-1} x_j^2 x_i = x_j^{-2}$, hence $A_n$ is closed under
conjugation and $A_n \normal G_n$. The relation $x_i^{-1} x_j^2 x_i
x_j^2 = 1$ is equivalent to $x_i x_j^2 = x_j^{-2} x_i$. Furthermore,
inverting both sides of this equation yields $x_j^{-2} x_i^{-1} =
x_i^{-1} x_j^2$, and multiplying this on the right and left by $x_i$
yields $x_i x_j^{-2} = x_j^2 x_i$.  Thus we may write

\[x_i^2 x_j^2 = x_i (x_i x_j^2) = (x_i x_j^{-2}) x_i = x_j^2 x_i^2.\]

Therefore, all of the generators of $A_n$ commute, and it remains only to show that the generators of $A_n$ are linearly independent.
Suppose by way of contradiction that the generators satisfy the
relation $(x_1^2)^{m_1} \dots (x_n)^{m_n}$, where the $m_i$ are
integers, not all zero.  Without loss of generality, we may assume
$m_1 \ne 0$.
Let $F_n$ denote the free group on $\{x_1,\dots,x_n\}$,
let $D= \langle a,b \mid a^2,b^2\rangle$ denote the infinite dihedral
group, and consider the epimorphism $\theta \colon F_n \to D$ defined
by $\theta x_1 = ba$ and $\theta x_i = b$ for all $i \ge 2$.
Since $\theta(x_i^{-1}x_j^2x_ix_j^2) = 1$ for all $i\ne j$,
we see that $\theta$ induces a homomorphism $\phi \colon G_n \to D$.
Then $\phi(x_1^2) = (ba)^2$ and $\phi(x_i^2) = b^2 = 1$ for $i\ge 2$.
Since $ba$ has infinite order, we have a contradiction and the result
follows.
\end{proof}

We need the following well known result.
\begin{lem} \label{L:freeproduct}
Let $n\in \mathbb{N}$, let $G = \langle x_1,\dots,x_n \mid
x_1^2,\dots x_n^2\rangle$, and let
\[N = \langle x_1x_2,\dots, x_{n-1}x_n\rangle.\]
Then $N\lhd G$, $|G/N| = 2$, and $N$ is freely
generated by $\{x_1x_2,\dots,x_{n-1}x_n\}$.  Furthermore any
nontrivial element of finite order in $G$ is conjugate to one of the
$x_i$.
\end{lem}
\begin{proof}
We have a well-defined surjective homomorphism $\theta \colon G \to
\mathbb{Z}/2\mathbb{Z}$ defined by $x_i \mapsto \bar{1}$.  Clearly
$N \subseteq \ker\theta$.  It is easily checked that $N\lhd G$ and
that $|G/N| \le 2$.  Therefore $N = \ker\theta$ and we deduce that
$|G/N| = 2$.  We now apply the Kurosh subgroup theorem \cite[Theorem
I.7.8]{Dicks89} and \cite[Exercise 3, p.~212]{Cohen89}.  We find that
$N$ is torsion free and any nontrivial finite subgroup of $G$ is
conjugate to one of the $x_i$.  Furthermore $N$ is free of rank
$n-1$, and the result follows.
\end{proof}

\begin{thm} \label{T:torsionfree}
$G_n$ is torsion free for all nonnegative integers $n$.
\end{thm}

\begin{proof}
Let $A_n$ be defined as in Lemma \ref{L:Free Abelian}.  Then $G_n/A_n
\cong \langle x_1, \dots, x_n \mid x_1^2, \dots x_n^2\rangle$.
By Lemma \ref{L:freeproduct},
this group has a free subgroup $N/A_n$ of index 2 freely generated by
the images of $x_1 x_2, \dots, x_{n-1} x_n$ in $G_n/A_n$, and any
nontrivial finite subgroup of $G_n / A_n$ is conjugate to $\gen{A_n
x_i} / A_n$ for some $i$. Therefore, to prove that $G_n$ is torsion
free, it will suffice to show that $\gen{A_n, x_i}$ is torsion free
for each $i$. By symmetry, we need only show this for some particular
$i$, so we set $i = 1$ and prove that $B := \gen{A_n, x_1} = \gen{x_1, x_2^2, \dots, x_n^2}$ is torsion free.

Now set $D = \gen{x_2^2, \dots, x_n^2}$. Since $A_n$ is torsion free
abelian and $D \leq A_n$, we have that $D$ is torsion free. It is
also an immediate corollary of the relations on $G_n$ that $D \normal
B$, and then we see that $B/D = \gen{Dx_1}$.  Since $A_n$ is free
abelian on $\{x_1^2,\dots,x_n^2\}$, we find that $Dx_1^2$ has
infinite order in $A_n/D$ and we deduce that $B/D
\cong \Z$, which is torsion free. Since both $D$ and $B/D$ are
torsion free, it follows that $B$ is torsion free, and from previous
arguments we conclude that $G_n$ is torsion free.
\end{proof}

\begin{prop} \label{P:Subgp}
If $1 \le m \le n$ are integers, then $G_m$ embeds in $G_n$.
\end{prop}

\begin{proof}
Let $A_m = \gen{x_1^2, \dots, x_m^2}$, a free abelian normal subgroup
of $G_m$ and $A_n = \gen{x_1^2, \dots, x_n^2}$, a free abelian normal
subgroup of $G_n$, by Lemma \ref{L:Free Abelian}.  Let $f\colon G_m \to
G_n$ be the natural homomorphism given by $f(x_i) = x_i$, and
consider the homomorphism $f^\prime \colon G_m / A_m \to G_n / A_n$ induced by $f$, so $f^\prime(A_mx_i) = A_nx_i$.
Since $G_m/A_m \cong \langle x_1,\dots, x_m \mid x_1^2, \dots, x_m^2 \rangle$ and
$G_n/A_n \cong \langle x_1,\dots, x_n \mid x_1^2, \dots, x_n^2 \rangle$, it is
clear that $f'$ is one-to-one and we deduce that $\ker f \subset A_m$,
so it remains to prove that $f$ is injective on $A_m$.  But $A_m$ is
free abelian on $\{x_1^2,\dots,x_m^2\}$ and $A_n$ is free abelian on
$\{x_1^2, \dots, x_n^2\}$ by Lemma \ref{L:Free Abelian},
and the result follows.
\end{proof}

\begin{cor}
For all integers $n \geq 2$, $G_n$ is a nonunique product group.
\end{cor}

\begin{proof}
It is well known that $G_2$ is a nonunique product group
\cite{Promislow88}.  Since $G_2$ is isomorphic to a subgroup of $G_n$
by Proposition \ref{P:Subgp}, the result follows.
\end{proof}

We now define the Hantzsche-Wendt groups studied by other
authors (called HW groups) as any group $\Gamma$ that is the
fundamental group of an orientable, flat $n$-manifold that has
holonomy group $\Z_2^{n-1}$.  We now state a few basic properties of
HW groups, see \cite[\S 2]{Putrycz07} and \cite[\S 9]{Szc12}
for further details.  We have that $n$ is odd,
$\Gamma$ is generated by elements $\beta_1,
\dots, \beta_n$, and that the translation subgroup of $\Gamma$ is
$\Lambda = \gen{\beta_1^2, \beta_2^2, \dots, \beta_n^2}$.
Then we can write $\beta_i = (B_i, b_i)$, where $b_i \in
\R^n$ such that $[b_i]_j \in \{\frac{1}{2}, 0\}$ (where $[b_i]_j$
denotes the $j$th entry of $b_i$) and the $B_i$ are $n
\times n$ diagonal matrices with all diagonal entries equal to $-1$ except the $(i,i)$ entry, which is equal to 1. Representing $\Gamma$ in this way, we obtain the product formula $(B, b)(C, c) = (BC, Bc + b)$. This product formula gives identities of the form $(A,a)(B,b)(C,c) = (ABC,ABc + Ab + a), (A,a)^{-1} = (A^{-1},-A^{-1}a)$ and $(A,a)^2 = (A^2,Aa+a).$ Identify $e_i$ with $(1,\epsilon_i)$, where $\epsilon_i$ is the $i$th standard basis element of $\mathbb{R}^n$. Then $\beta_i^2 = e_i$, and we have for all $i \not = j$ that
\begin{align*}
\beta_i^{-1} \beta_j^2 \beta_i &= (B_i^{-1},-B_i^{-1}b_i)
(1,e_j)(B_i, b_i)= (1,B_i^{-1}b_i + B_i^{-1}e_j - B_i^{-1}b_i)\\
&= (1,B_i^{-1} e_j) = (1,-e_j) = \beta_j^{-2}
\end{align*}
and therefore $\beta_i^{-1} \beta_j^2 \beta_i \beta_j^2 = 1$, so the $\beta_i$ satisfy the relations on $G_n$. A second consequence of this representation of $\Gamma$ is that $\Gamma \cap \R^n = \Lambda$.

\begin{thm}  \label{T:surjective}
Let $n\in \mathbb{N}$.  Then there is a surjective homomorphism
$\Phi \colon G_{n-1} \to \Gamma$ for all HW groups $\Gamma$
of dimension $n$, and $\ker\Phi$ is a free group.
\end{thm}

\begin{proof}
Let $\Gamma$ be an HW group of dimension $n$.
We know that $n$ must be odd, so write
$n = 2k+1$ for some nonnegative integer $k$.
Define the map $\Phi \colon G_{n-1} \to \Gamma$ by $\Phi(x_i) = \beta_i$, which is a homomorphism since the generators of $\Gamma$ satisfy $\beta_i^{-1} \beta_j^2 \beta_i \beta_j^2 = 1$ whenever $i \not = j$. It follows from the definition of $\Phi$ that $\gen{\beta_1, \dots, \beta_{n-1}} \subseteq \Phi(G_{n-1})$. To show that $\Phi$ is onto $\Gamma$, it suffices to show that $\beta_n \in \gen{\beta_1, \dots, \beta_{n-1}}$.

Set $I = \{1, 2, \dots, 2k\}$. Then by
\cite[Proposition 2.1]{Putrycz07}, we
have that there exists $j \not \in I$ such that $|\{i \in I \mid [b_i]_j
= \frac{1}{2} \}|$ is odd. Since $I$ is defined as a subset of $\{1,
2, \dots, n\}$, it follows that $j = 2k+1 = n$. Now let $\pi \colon I \to I$ be a permutation, and define $P = \prod\limits_{i = 1}^{2k} \beta_{\pi(i)} = \prod\limits_{i=1}^{2k} (B_{\pi(i)}, b_{\pi(i)})$. From repeated application of the multiplication of these pairs, we obtain
\begin{multline*}
P = (B_{\pi(1)} B_{\pi(2)} \cdots B_{\pi(2k)}, b_{\pi(1)} +
B_{\pi(1)} b_{\pi(2)} +\\ B_{\pi(1)} B_{\pi(2)} b_{\pi(3)} + \dots +
B_{\pi(1)} \cdots B_{\pi(2k-1)} b_{\pi(2k)}).
\end{multline*}
Define $v_x = \bigg{(} \prod\limits_{i=1}^{x-1} B_{\pi(i)} \bigg{)} b_{\pi(x)}$. Now the $B_i$ commute (since they are diagonal matrices) so we have $\prod\limits_{i=1}^{2k} B_{\pi(i)} = \prod\limits_{i=1}^{2k} B_i = B_n$, and therefore $P = \bigg{(}B_n, \sum\limits_{i=1}^{2k} v_i\bigg{)}$. We now wish to show that for some permutation $\pi$, the $j$th entry of $\sum\limits_{i=1}^{2k} v_i$ is equal to $\frac{1}{2}$. Then for each $v_i$ we have
\[[v_i]_n = [B_{\pi(1)} \cdots B_{\pi(i-1)} b_{\pi(i)}]_n = (-1)^{i-1} [b_{\pi(i)}]_n\]
since the $(n,n)$th entry of $B_i$ is equal to $-1$ for all $i \in I$. Therefore we have
\[\bigg{[} \sum\limits_{i=1}^{2k} v_i \bigg{]}_n = \sum\limits_{i=1}^{2k} (-1)^{i-1} [b_{\pi(i)}]_n = [b_{\pi(1)} - b_{\pi(2)} + \dots - b_{\pi(2k)}]_n.\]
Earlier it was shown that if $J = \{i \in I \mid [b_i]_n = \frac{1}{2} \}$, then $|J|$ is odd. If $\pi$ satisfies $\pi(\{1, 2, \dots, |J|\}) = J$, which we may assume without loss of generality, then this implies $[b_{\pi(1)} - b_{\pi(2)} + \dots - b_{\pi(2k)}]_n = \frac{1}{2}$. Therefore we may write
\[P = (B_n, a_1 e_1 + a_2 e_2 + \dots + a_n e_n)\]
where $2a_i \in \Z$ for each $i$ and $a_n = \pm \frac{1}{2}$. Therefore,
\[P^2 = (1, B_n(a_1 e_1 + \dots + a_n e_n) + a_1 e_1 + \dots + a_n e_n) = \pm e_n.\]
Thus $e_n \in \gen{\beta_1, \dots, \beta_{n-1}}$, and it
follows that $\Lambda \subset \gen{\beta_1, \dots, \beta_{n-1}}$.
Now, $\beta_1 \beta_2 \dots \beta_n \in \R^n$ because $B_1 B_2 \dots
B_n = I$, and since $\Gamma \cap \R^n = \Lambda$, we have $\beta_1
\beta_2 \dots \beta_n \in \Lambda \subset \gen{\beta_1, \dots,
\beta_{n-1}}$. Thus, $\beta_n \in \gen{\beta_1, \dots, \beta_{n-1}}$, and therefore $\Phi$ must be onto $\Gamma$.

Write $A_{n-1} = \langle x_1^2,\dots,x_{n-1}^2\rangle$ and $K = \ker\Phi$.
Since $A_{n-1}$ is abelian by Lemma \ref{L:Free Abelian}, $A_{n-1}\lhd G_{n-1}$
and $\beta_i^2 = e_i$, we see that $K \cap A_{n-1} = 1$ and therefore $K$ is
isomorphic to a subgroup of $G_{n-1}/A_{n-1} \cong \langle x_1,\dots,
x_{n-1} \mid x_1^2,\dots,x_{n-1}^2 \rangle$.
Furthermore $K$ is torsion free by Theorem \ref{T:torsionfree}, and
we conclude that $K$ is free by the Kurosh subgroup theorem
\cite[Theorem I.7.8]{Dicks89}.
\end{proof}

\begin{rem*}
We note that the groups $G_n$ satisfy the Kaplansky zero divisor
conjecture.  In fact if $k$ is any field, then $kG_n$ can be embedded
in a division ring.  One way to see this is as follows.  Since $G_m$
embeds in $G_n$ for $m\le n$ by Lemma \ref{P:Subgp}, we may assume
that $n$ is odd.  Also $G_n$ is torsion free by Theorem
\ref{T:torsionfree}.  Then by Theorem \ref{T:surjective}, there is a
normal free subgroup $K$ of $G_n$ such that $G_n/K$ is isomorphic to
an HW group of dimension $n$.  Since an HW group is virtually abelian,
we may apply \cite[Theorem 1.5]{Linnell93}, at least in the case
$k=\mathbb{C}$.  However the arguments of \cite[\S 4]{Linnell93} apply
for any field $k$.
\end{rem*}

\bibliographystyle{plain}

\end{document}